\documentclass[12pt,leqno]{article}
\usepackage{amssymb,graphics,latexsym}
\usepackage{graphicx}
\usepackage{amsmath,amsthm}

\theoremstyle{plain}
\newtheorem{theorem}{Theorem}[section]

\newtheorem{corollary}[theorem]{Corollary}
\newtheorem{lemma}[theorem]{Lemma}
\newtheorem{proposition}[theorem]{Proposition}

\theoremstyle{remark}

\newtheorem{example}[theorem]{Example}

\newcommand{\re}{{\rm Re\, }}

\newcommand{\tr}{{\rm tr\, }}
\newcommand{\dia}{{\rm diag\, }}

\newcommand{\rank}{{\rm rank\, }}

\begin{document}


\thispagestyle{empty}

\begin{center}
{\Large \bf Numerical Radii for Tensor Products of Matrices}
\end{center}

\vspace*{3mm}

\centerline{\bf Hwa-Long Gau$^{a\ast}$, \hspace*{.3cm} \bf Kuo-Zhong Wang$^{b}$,\hspace*{.3cm}\bf Pei Yuan Wu$^{b}$}

\vspace*{3mm}

\noindent{\small
${ }^{a}$Department of Mathematics, National Central University, Chung-Li 32001, Taiwan \\
${ }^{b}$Department of Applied Mathematics, National Chiao
Tung University, Hsinchu 30010,\\ \hspace*{1mm} Taiwan}

\vspace*{7mm}

\noindent
{\bf Abstract.} For $n$-by-$n$ and $m$-by-$m$ complex matrices $A$ and $B$, it is known that the inequality $w(A\otimes B)\le\|A\|w(B)$ holds, where $w(\cdot)$ and $\|\cdot\|$ denote, respectively, the numerical radius and the operator norm of a matrix. In this paper, we consider when this becomes an equality. We show that (1) if $\|A\|=1$ and $w(A\otimes B)=w(B)$, then either $A$ has a unitary part or $A$ is completely nonunitary and the numerical range $W(B)$ of $B$ is a circular disc centered at the origin, (2) if $\|A\|=\|A^k\|=1$ for some $k$, $1\le k<\infty$, then $w(A)\ge\cos(\pi/(k+2))$, and, moreover, the equality holds if and only if $A$ is unitarily similar to the direct sum of the $(k+1)$-by-$(k+1)$ Jordan block $J_{k+1}$ and a matrix $B$ with $w(B)\le\cos(\pi/(k+2))$, and (3) if $B$ is a nonnegative matrix with its real part (permutationally) irreducible, then $w(A\otimes B)=\|A\|w(B)$ if and only if either $p_A=\infty$ or $n_B\le p_A<\infty$ and $B$ is permutationally similar to a block-shift matrix
\[\left[
    \begin{array}{cccc}
      0 & B_1 &   &   \\
        & 0 & \ddots &   \\
        &   & \ddots & B_k \\
        &   &   & 0 \\
    \end{array}
  \right]\]
with $k=n_B$, where $p_A=\sup\{\ell\ge 1: \|A^{\ell}\|=\|A\|^{\ell}\}$ and $n_B=\sup\{\ell\ge 1 : B^{\ell}\neq 0\}$.

\vspace{5mm}



\vspace{8mm}

\hrule
\vspace{2mm}

\noindent
${}^*$Corresponding author. Email: hlgau@math.ncu.edu.tw

\newpage

\section{Introduction and Preliminaries}

For any $n$-by-$n$ complex matrix $A$, its \emph{numerical range} $W(A)$ is, by definition, the subset $\{\langle Ax, x\rangle : x\in \mathbb{C}^n, \|x\|=1\}$ of the complex plane $\mathbb{C}$, where $\langle\cdot, \cdot\rangle$ and $\|\cdot\|$ denote the standard inner product and its associated norm in $\mathbb{C}^n$, respectively. The \emph{numerical radius} $w(A)$ of $A$ is $\max\{|z| : z\in W(A)\}$. It is known that $W(A)$ is a nonempty compact convex subset of $\mathbb{C}$, and $w(A)$ satisfies $\|A\|/2\le w(A)\le \|A\|$, where $\|A\|$ denotes the usual operator norm of $A$. For other properties of the numerical range and numerical radius, the reader may consult \cite{7}, \cite[Chapter 22]{9} or \cite[Chapter 1]{12}.

\vspace{4mm}

The \emph{tensor product} (or \emph{Kronecker product}) $A\otimes B$ of an $n$-by-$n$ matrix $A=[a_{ij}]_{i,j=1}^n$ and an $m$-by-$m$ matrix $B$ is the $(mn)$-by-$(mn)$ matrix
\[\left[
    \begin{array}{ccc}
      a_{11}B & \cdots & a_{1n}B \\
      \vdots &   & \vdots \\
      a_{n1}B & \cdots & a_{nn}B \\
    \end{array}
  \right].\]
It is known that $A\otimes B$ and $B\otimes A$ are unitarily similar and $\|A\otimes B\|=\|A\|\cdot\|B\|$. Other properties of the tensor product can be found in \cite[Chapter 4]{12}.

\vspace{4mm}

The main concern of this paper is the relations between the numerical radius of $A\otimes B$ and those of $A$ and $B$. For one direction, we have $w(A\otimes B)\le\min\{\|A\|w(B), \|B\|w(A)\}$. This can be proven by using the unitary dilation of contractions, as to be done below. On the other hand, we also have $w(A\otimes B)\ge w(A)w(B)$. We are interested in when these become equalities. In the present paper, we obtain various conditions, necessary or sufficient, for $w(A\otimes B)=\|A\|w(B)$ to hold. The discussions on the equality $w(A\otimes B)=w(A)w(B)$ will be the subject of a subsequent paper of ours.

\vspace{4mm}

For the ease of exposition, we introduce two indices for an $n$-by-$n$ matrix $A$: the \emph{power norm index} $p_A$ and \emph{nilpotency index} $n_A$ of $A$. They are defined, respectively, by
\[p_A=\sup\{k\ge 1 : \|A^k\|=\|A\|^k\}\]
and
\[n_A=\left\{\begin{array}{ll} \sup\{k\ge 1 : A^k\neq 0_n\} & \ \ \ \mbox{if } \ A\neq 0_n,\\ 0 & \ \ \ \mbox{if } \ A=0_n,\end{array}\right.\]
where $0_n$ denotes the $n$-by-$n$ zero matrix.

\vspace{4mm}

We start in Section 2 by proving that if $\|A\|=1$ and $w(A\otimes B)=w(B)$, then either $A$ has a unitary part or $A$ is completely nonunitary and $W(B)$ is a circular disc centered at the origin (Theorem \ref{2.2}). The proof depends on the dilation of $A$ to a direct sum of $S_{\ell}$-matrices with $\ell\le n$, the Poncelet property of the numerical ranges of matrices of the latter class, and Anderson's theorem on the circular disc numerical range. As a by-product, we obtain a lower bound for $w(A)$ when $A$ satisfies $\|A\|=\|A^k\|=1$ for some $k$, $1\le k< n$: $w(A)\ge\cos(\pi/(k+2))$, and determine exactly when this bound is attained: this is the case if and only if $A$ is unitarily similar to $J_{k+1}\oplus B$, where $J_{k+1}$ is the $(k+1)$-by-$(k+1)$ \emph{Jordan block}
\[\left[
    \begin{array}{cccc}
      0 & 1 &   &   \\
        & 0 & \ddots &   \\
        &   & \ddots & 1 \\
        &   &   & 0 \\
    \end{array}
  \right]\]
and $B$ is a finite matrix with $w(B)\le\cos(\pi/(k+2))$ (Theorem \ref{2.10}). This generalizes the classical result of Willams and Crimmins \cite{17} for $k=1$. We conclude Section 2 with a result on nilpotent contractions, namely, we prove that if $A$ is an $n$-by-$n$ matrix with $\|A\|=1$, then a necessary and sufficient condition for $p_A=n_A<\infty$ to hold is that $A$ be unitarily similar to a direct sum $J_{k+1}\oplus B$, where $k=p_A$ and $B^{k+1}=0$ (Theorem \ref{2.13}).

\vspace{4mm}

Finally, in Section 3, we consider $B$ to be a nonnegative matrix with $\re B$ ($=(B+B^*)/2$) (permutationally) irreducible. We obtain in Theorem \ref{3.1} a complete characterization for $w(A\otimes B)=\|A\|w(B)$, namely, this is the case if and only if either $p_A=\infty$ or $n_B\le p_A<\infty$ and $B$ is permutationally similar to a block-shift matrix of the form
\[\left[
    \begin{array}{cccc}
      0 & B_1 &   &   \\
        & 0 & \ddots &   \\
        &   & \ddots & B_k \\
        &   &   & 0 \\
    \end{array}
  \right]\]
with $k=n_B$.

\vspace{4mm}

As was mentioned before, the inequality $w(A\otimes B)\le\|A\|w(B)$ for $n$-by-$n$ and $m$-by-$m$ matrices $A$ and $B$ is known. It is a consequence of \cite[Theorem 3.4]{10} because $A\otimes B$ is the product of $A\otimes I_m$ and $I_n\otimes B$, and the latter two matrices \emph{doubly commute}, that is, $A\otimes I_m$ commutes with both $I_n\otimes B$ and its adjoint $I_n\otimes B^*$. Here we give a simple proof based on the unitary dilation of contractions.

\vspace{3mm}

\begin{proposition}\label{1.1}
If $A$ and $B$ are $n$-by-$n$ and $m$-by-$m$ matrices, respectively, then $w(A\otimes B)\le\min\{\|A\|w(B), \|B\|w(A)\}$.
\end{proposition}

\vspace{1pt}

\begin{proof}
We need only prove that $w(A\otimes B)\le \|A\|w(B)$, and may assume that $\|A\|=1$. Then the $(2n)$-by-$(2n)$ matrix
\[U=\left[
    \begin{array}{cc}
      A & (I_n-AA^*)^{1/2} \\
      (I_n-A^*A)^{1/2} & -A^* \\
    \end{array}
  \right]\]
is unitary. Let $U$ be unitarily similar to the diagonal matrix $\dia(u_1, \ldots, u_{2n})$, where $|u_j|=1$ for all $j$. Then
\[w(A\otimes B)\le w(U\otimes B) = w(\sum_{j=1}^{2n}\oplus u_jB)=\max_j w(u_jB)=w(B)=\|A\|w(B).\qedhere\]
\end{proof}

\vspace{1mm}

We conclude this section with some basic properties of the indices $p_A$ and $n_A$ of a matrix $A$.

\vspace{3mm}

\begin{proposition}\label{1.2}
Let $A$ be an $n$-by-$n$ matrix. Then
\begin{enumerate}
\item[\rm (a)] $1\le p_A\le n-1$ or $p_A=\infty$,
\item[\rm (b)] $p_A=n-1$ if and only if $A$ is a nonzero multiple of a $S_n$-matrix, and
\item[\rm (c)] the following conditions are equivalent:
\begin{enumerate}
\item[\rm (1)] $p_A=\infty$,
\item[\rm (2)] $\|A\|=\rho(A)$,
\item[\rm (3)] $\|A\|=w(A)$,
\end{enumerate}
and if $\|A\|=1$, then the above are also equivalent to
\begin{enumerate}
\item[\rm (4)] $A$ has a unitary part.
\end{enumerate}
\end{enumerate}
\end{proposition}

\vspace{3mm}

Here $\rho(A)$ denotes the \emph{spectral radius} of $A$, that is, $\rho(A)$ is the maximum modulus of eigenvalues of $A$.

\vspace{3mm}

Recall that an $n$-by-$n$ matrix $A$ is of \emph{class} $S_n$ if it is a contraction ($\|A\|\le 1$), its eigenvalues are all in $\mathbb{D}\equiv\{z\in\mathbb{C} : |z|<1\}$, and $\rank(I_n-A^*A)=1$. Any contraction $A$ is unitarily similar to the direct sum of a unitary matrix $U$, called the \emph{unitary part} of $A$, and a completely nonunitary contraction $A'$, called the \emph{c.n.u. part} of $A$. The latter means that $A'$ is not unitarily similar to any direct sum with a unitary summand.

\vspace{3mm}

\begin{proof}[Proof of Proposition $\ref{1.2}$]
(a) was obtained by Pt\'{a}k in 1960 (cf. \cite[Theorem 2.1]{15}) and (b) was proven in \cite[Theorem 3.1]{4}. As for (c), the implication (1) $\Rightarrow$ (2) is by \cite[Problem 88]{9}, (2) $\Rightarrow$ (3) by the known inequalities $\rho(A)\le w(A)\le\|A\|$, (3) $\Rightarrow$ (2) by \cite[Problem 218 (b)]{9}, and (2) $\Rightarrow$ (1) by the inequalities $\rho(A)\le\|A^k\|^{1/k}\le\|A\|$ for all $k \ge 1$. If $\|A\|=\rho(A)=1$, then, letting $\lambda$ be an eigenvalue of $A$ with $|\lambda|=1$, we have the unitary similarity of $A$ and a matrix of the form
{\footnotesize$\left[   \begin{array}{cc}     \lambda & B \\     0 & C \\   \end{array} \right]$}.
Since $\|A\|=|\lambda|=1$ implies that $B=0$, $A$ is unitarily similar to $[\lambda]\oplus C$ and thus has a unitary part. This proves (2) $\Rightarrow$ (4). That (4) $\Rightarrow$ (2) is trivial.
\end{proof}

\vspace{1mm}

\begin{proposition}\label{1.3}
Let $A$ be an $n$-by-$n$ matrix. Then
\begin{enumerate}
\item[\rm (a)] $0\le n_A\le n-1$ or $n_A=\infty$,
\item[\rm (b)] $n_A=n-1$ if and only if $A$ is similar to the $n$-by-$n$ Jordan block $J_n$,
\item[\rm (c)] $n_A=\infty$ if and only if $A$ is not nilpotent, and
\item[\rm (d)] $p_A\le n_A$ for $A\neq 0_n$.
\end{enumerate}
\end{proposition}

\vspace{3mm}

We omit its easy proofs.

\vspace{3mm}

In the following, we use $\sigma(A)$ to denote the \emph{spectrum} of $A$, that is, $\sigma(A)$ is the set of eigenvalues of $A$. An $n$-by-$n$ matrix $A$ is a \emph{dilation} of an $m$-by-$m$ matrix $B$ (or $B$ is a \emph{compression} of $A$) if there is an $n$-by-$m$ matrix $V$ such that $B=V^*AV$ and $V^*V=I_m$. This is equivalent to $A$ being unitarily similar to a matrix of the form {\footnotesize$\left[ \begin{array}{cc} B & * \\ * & * \\ \end{array} \right]$}.

\vspace{5mm}

\section{Contractions}

We start with a simple condition which yields the equality $w(A\otimes B)=\|A\|w(B)$.

\vspace{1mm}

\begin{lemma}\label{2.1}
If $A$ is an $n$-by-$n$ matrix with $p_A=\infty$, then $w(A\otimes B)=\|A\|w(B)$ for any $m$-by-$m$ matrix $B$. In particular, this is the case for $A$ a contraction with a unitary part.
\end{lemma}

\vspace{1pt}

\begin{proof}
Since $p_A=\infty$ implies, by Proposition \ref{1.2} (c), that $\|A\|=w(A)$. If $\lambda$ is a number in $W(A)$ with $|\lambda|=w(A)$, then $|\lambda|=\|A\|$. Since $A$ is unitarily similar to a matrix of the form {\footnotesize$\left[\begin{array}{cc} \lambda & * \\ * & * \\ \end{array} \right]$}, we have the unitary similarity of $A\otimes B$ and {\footnotesize$\left[\begin{array}{cc} \lambda B & * \\ * & * \\ \end{array} \right]$}. It follows that $\|A\|w(B)=w(\lambda B)\le w(A\otimes B)$. On the other hand, we also have $w(A\otimes B)\le\|A\|w(B)$ by Proposition \ref{1.1}. Thus $w(A\otimes B)=\|A\|w(B)$ holds.
\end{proof}

\vspace{1mm}

The next theorem is one of the main results of this section. It gives a necessary condition for the equality $w(A\otimes B)=\|A\|w(B)$.

\vspace{1mm}

\begin{theorem}\label{2.2}
Let $A$ and $B$ be $n$-by-$n$ and $m$-by-$m$ matrices, respectively. If $\|A\|=1$ and $w(A\otimes B)=w(B)$, then either $A$ has a unitary part or A is c.n.u. and $W(B)$ is a circular disc centered at the origin.
\end{theorem}

\vspace{1mm}

We first prove this for the case when $A$ is an $S_n$-matrix. The numerical ranges of such matrices are known to have the \emph{Poncelet property}, namely, if $A$ is of class $S_n$, then, for any point $\lambda$ on the unit circle $\partial \mathbb{D}$, there is a unique (up to unitary similarity) $(n+1)$-by-$(n+1)$ unitary dilation $U$ of $A$ such that $\lambda$ is an eigenvalue of $U$ and each edge of the $(n+1)$-gon $\partial W(U)$ intersects $W(A)$ at exactly one point (cf. \cite[Theorem 2.1 and Lemma 2.2]{2}).

\vspace{1mm}

\begin{lemma}\label{2.3}
Let $A$ be an $S_n$-matrix and $B$ an $m$-by-$m$ matrix. If $w(A\otimes B)=w(B)$, then $W(B)$ is a circular disc centered at the origin.
\end{lemma}

\vspace{1pt}

\begin{proof}
Let $U_1, \ldots, U_{m+1}$ be $(n+1)$-by-$(n+1)$ unitary dilations of $A$ with $\sigma(U_i)\cap\sigma(U_j)=\emptyset$ for all $i$ and $j$, $1\le i\neq j\le m+1$. We may assume that $U_j=\dia(\lambda_{1 j}, \ldots, \lambda_{n+1, j})$ for each $j$, where $|\lambda_{i j}|=1$ for all $i$ and $j$. Let $V_j$ be an $(n+1)$-by-$n$ matrix such that $A=V_j^*U_jV_j$ and $V_j^*V_j=I_n$ for each $j$. Since $\|A\|=1$ and
\[ w(A\otimes \lambda B)=w(A\otimes B)=w(B)=w(\lambda B)\]
for any $\lambda$, $|\lambda|=1$, we may further assume that $w(B)$ is in $W(A\otimes B)$. Let $x$ be a unit vector in $\mathbb{C}^n\otimes \mathbb{C}^m$ such that $\langle (A\otimes B)x, x\rangle=w(B)$. We decompose $(V_j\otimes I_m)x$ as $y_{1j}\oplus\cdots\oplus y_{n+1, j}$ with $y_{ij}$, $1\le i\le n+1$, in $\mathbb{C}^m$ for each $j$. Then
\begin{align*}
w(B) &= \langle (A\otimes B)x, x\rangle\\
&= \langle (U_j\otimes B)(V_j\otimes I_m)x, (V_j\otimes I_m)x\rangle\\
&= \langle (\lambda_{1 j}B\oplus\cdots\oplus\lambda_{n+1, j}B)(y_{1 j}\oplus\cdots\oplus y_{n+1, j}), y_{1 j}\oplus\cdots\oplus y_{n+1, j}\rangle\\
&= \sum_{i=1}^{n+1}\langle \lambda_{i j}By_{i j}, y_{i j}\rangle\\
&\le \sum_{i=1}^{n+1}|\langle By_{i j}, y_{i j}\rangle|.
\end{align*}
Letting $\eta_{ij}=\langle B(y_{i j}/\|y_{i j}\|), y_{i j}/\|y_{i j}\|\rangle$ for each $y_{ij}\neq 0$, we obtain
\[w(B)=\sum_{y_{ij}\neq 0}\lambda_{ij}\|y_{ij}\|^2\eta_{ij}\le\sum_{y_{ij}\neq 0}\|y_{ij}\|^2|\eta_{ij}|\le\sum_{y_{ij}\neq 0}\|y_{ij}\|^2 w(B)=w(B)\]
since
\[\sum_{i}\|y_{ij}\|^2=\|(V_j\otimes I_m)x\|^2=\|x\|^2=1.\]
Thus we have equalities throughout the above sequence, which yields that $w(B)=\lambda_{ij}\eta_{ij}$ for $y_{ij}\neq 0$. Since $\sum_i\|y_{ij}\|^2=1$, this must hold for at least one $i$, say, $i_j$. Hence $\overline{\lambda}_{i_j j}w(B)=\eta_{i_j j}$ is in $\partial W(B)$ for each $j$. Note that such $\overline{\lambda}_{i_j j}w(B)$'s, $1\le j\le m+1$, are distinct from each other by our assumption on the disjointness of the spectra of the $U_j$'s. This shows that the boundary of $W(B)$ and the circle $|z|=w(B)$ intersect at at least $m+1$ points. Since $W(B)$ is contained in $\{z\in\mathbb{C} : |z|\le w(B)\}$, we apply Anderson's theorem (cf. \cite[Theorem]{3} or \cite{20}) to infer that $W(B)=\{z\in\mathbb{C}: |z|\le w(B)\}$.
\end{proof}

\vspace{2mm}

\begin{proof}[Proof of Theorem $\ref{2.2}$]
We assume that $A$ is c.n.u. Then $A$ can be dilated to the direct sum $A'\oplus\cdots\oplus A'$ of $\rank(I_n-A^*A)$ many copies of some $S_{\ell}$-matrix $A'$ with $\ell\le n$ (cf. \cite[Theorem 1.4]{18} or \cite[Lemma 3 (a)]{21}). Hence $A\otimes B$ dilates to $(A'\oplus\cdots\oplus A')\otimes B=(A'\otimes B)\oplus\cdots\oplus (A'\otimes B)$. We have
\[ w(B)=w(A\otimes B)\le w((A'\otimes B)\oplus\cdots\oplus (A'\otimes B))= w(A'\otimes B)\le\|A'\|w(B)=w(B).\]
Thus $w(A'\otimes B)=w(B)$. It follows from Lemma \ref{2.3} that $W(B)$ is a circular disc centered at the origin.
\end{proof}

\vspace{3mm}

An easy consequence of Theorem \ref{2.2} is that the converse of Lemma \ref{2.1} is also true.

\vspace{3mm}

\begin{corollary}\label{2.4}
For an $n$-by-$n$ matrix $A$, the equality $w(A\otimes B)=\|A\|w(B)$ holds for all matrices $B$ if and only if $p_A=\infty$.
\end{corollary}

\vspace{1pt}

\begin{proof}
For the necessity, assume that $\|A\|=1$ and let $B$ be any matrix with its numerical range not a circular disc centered at the origin. Theorem \ref{2.2} yields that $A$ has a unitary part. Then $p_A=\infty$ follows immediately.
\end{proof}

\vspace{3mm}

In Theorem \ref{2.2}, if $B$ is the Jordan block $J_m$, then we have the following characterizations for $w(A\otimes B)=\|A\|w(B)$.

\vspace{3mm}

\begin{theorem}\label{2.5}
Let $A$ be an $n$-by-$n$ matrix with $\|A\|=1$. Then the following conditions are equivalent:
\begin{enumerate}
\item[\rm (a)] $W(A\otimes J_m)=W(J_m)$,
\item[\rm (b)] $w(A\otimes J_m)=w(J_m)$,
\item[\rm (c)] $A\otimes J_m$ is unitarily similar to $J_m\oplus B$ for some matrix $B$ with $w(B)\le w(J_m)$, and
\item[\rm (d)] $\|A^{m-1}\|=1$.
\end{enumerate}
If, in addition, $n=m$, then the above conditions are also equivalent to
\begin{enumerate}
\item[\rm (e)] either $A$ has a unitary part or $A$ is of class $S_n$, and
\item[\rm (f)] $p_A=\infty$ or $n-1$.
\end{enumerate}
\end{theorem}

\vspace{3mm}

Note that $W(J_m)=\{z\in\mathbb{C}: |z|\le\cos(\pi/(m+1))\}$ (cf. \cite[Proposition 1]{8}).

\vspace{3mm}

\begin{proof}[Proof of Theorem $\ref{2.5}$]
The implication (a) $\Rightarrow$ (b) is trivial. To prove (b) $\Rightarrow$ (c), note that $(A\otimes J_m)^m=A^m\otimes J_m^m=0_{nm}$ and $\|A\otimes J_m\|=\|A\|\|J_m\|=1$. If $x$ is a unit vector in $\mathbb{C}^n\otimes \mathbb{C}^m$ such that $|\langle (A\otimes J_m)x, x\rangle|=w(A\otimes J_m)$, then $w(A\otimes J_m)=w(J_m)=\cos(\pi/(m+1))$ implies that the subspace $K$ of $\mathbb{C}^n\otimes \mathbb{C}^m$ generated by the vectors $x, (A\otimes J_m)x, \ldots, (A\otimes J_m)^{m-1}x$ is reducing for $A\otimes J_m$, and the restriction of $A\otimes J_m$ to $K$ is unitarily similar to $J_m$ (cf. \cite[Theorem 1 (2)]{8}). Hence $A\otimes J_m$ is unitarily similar to $J_m\oplus B$, where $B$ is the restriction of $A\otimes J_m$ to $K^{\perp}$. We obviously have $w(B)\le w(A\otimes J_m)=w(J_m)$.

\vspace{3mm}

For (c) $\Rightarrow$ (d), note that $A^{m-1}\otimes J_m^{m-1}$ is unitarily similar to $J_m^{m-1}\oplus B^{m-1}$ under (c). Hence
\[\|A^{m-1}\|=\|A^{m-1}\otimes J_m^{m-1}\|=\|J_m^{m-1}\oplus B^{m-1}\|=\max\{\|J_m^{m-1}\|, \|B^{m-1}\|\}=1.\]

\vspace{3mm}

To prove (d) $\Rightarrow$ (c), let $x$ be a unit vector in $\mathbb{C}^n$ such that $\|A^{m-1}x\|=1$. Then $\|A^{m-j}x\|=1$ for all $j$, $1\le j\le m$. Let $\{e_1, \ldots, e_m\}$ be the standard basis for $\mathbb{C}^m$, let $x_j=A^{m-j}x\otimes e_j$, $1\le j\le m$, and let $K$ be the subspace of $\mathbb{C}^n\otimes \mathbb{C}^m$ generated by $x_1, \ldots, x_m$. Then $(A\otimes J_m)x_1=0$ and $(A\otimes J_m)x_j=x_{j-1}$ for $2\le j\le m$. Since $\{x_1, \ldots, x_m\}$ is an orthonormal basis of $K$, this shows that $(A\otimes J_m)K\subseteq K$ and the restriction of $A\otimes J_m$ to $K$ is unitarily similar to $J_m$. On the other hand, it follows from $\|A\otimes J_m\|=\|A\|\|J_m\|=1$ and
\[ (A\otimes J_m)^*x_m=(A^*\otimes J_m^*)(x\otimes e_m)=(A^*x)\otimes(J_m^*e_m)=(A^*x)\otimes 0=0\]
that $K$ is reducing for $A\otimes J_m$, and hence $A\otimes J_m$ is unitarily similar to $J_m\oplus B$, where $B$ is the restriction of $A\otimes J_m$ to $K^{\perp}$. Obviously, we have
\[ w(B)\le w(A\otimes J_m)\le\|A\|w(J_m)=w(J_m).\]

\vspace{3mm}

To prove (c) $\Rightarrow$ (a), note that the unitary similarity of $J_m$ and $e^{i\theta}J_m$ for all real $\theta$ implies the same for $A\otimes J_m$ and $e^{i\theta}(A\otimes J_m)$. Thus $W(A\otimes J_m)$ is a circular disc centered at the origin. (c) implies that $w(A\otimes J_m)=w(J_m)$, which means that the radii of the two circular discs $W(A\otimes J_m)$ and $W(J_m)$ are equal. Therefore, $W(A\otimes J_m)=W(J_m)$ holds.

\vspace{3mm}

Now assume that $n=m$ and that $\|A^{n-1}\|=1$. If $\|A^n\|=1$, then $p_A=\infty$ and hence $A$ has a unitary part by Proposition \ref{1.2} (a) and (c). On the other hand, if $\|A^n\|<1$, then $A$ is of class $S_n$ by \cite[Theorem 3.1]{4}. This shows that (d) $\Rightarrow$ (e). Next, if (e) is true, then $p_A=\infty$ or $n-1$ depending on whether $A$ has a unitary part or $A$ is of class $S_n$ (cf. \cite[Theorem 3.1]{4} for the latter). This proves (f). Finally, if $p_A=\infty$, then $\|A^k\|=1$ for all $k\ge 1$, and, in particular, $\|A^{n-1}\|=1$. On the other hand, if $p_A=n-1$, then $\|A^{n-1}\|=\|A\|^{n-1}=1$. This proves (f) $\Rightarrow$ (d).
\end{proof}

\vspace{3mm}

The next proposition gives a characterization of $w(A\otimes B)=\|A\|w(B)$ when $B$ is of class $S_m$.

\vspace{3mm}

\begin{proposition}\label{2.6}
Let $A$ be an $n$-by-$n$ matrix with $\|A\|=1$, and $B$ be an $S_m$-matrix. Then $w(A\otimes B)=w(B)$ if and only if either $A$ has a unitary part or $A$ is c.n.u., $\|A^{m-1}\|=1$ and $B$ is unitarily similar to $J_m$.
\end{proposition}

\vspace{3mm}

Its proof depends on a special property of $S_n$-matrices. The following lemma is from \cite[Lemma 5]{19}. Here we give a shorter geometric proof.

\vspace{3mm}

\begin{lemma}\label{2.7}
Let $A$ be an $S_n$-matrix. Then $W(A)$ is a circular disc centered at the origin if and only if $A$ is unitarily similar to $J_n$.
\end{lemma}

\vspace{1pt}

\begin{proof}
If $W(A)$ is as asserted, then the Poncelet property of $W(A)$ says that it is circumscribed by $(n+1)$-gons with vertices on the unit circle. As the circular disc $\{z\in\mathbb{C}: |z|\le\cos(\pi/(n+1))\} (=W(J_n))$ is circumscribed by any regular $(n+1)$-gon on the unit circle, if the radius of $W(A)$ is not equal to $\cos(\pi/(n+1))$, then we infer from a geometrical consideration that $W(A)$ cannot have the Poncelet property. Thus $W(A)$ must equal $W(J_n)$. The unitary similarity of $A$ and $J_n$ then follows from \cite[Theorem 3.2]{2}. The converse is trivial.
\end{proof}

\vspace{1mm}

\begin{proof}[Proof of Proposition $\ref{2.6}$]
If $w(A\otimes B)=w(B)$, then, by Theorem \ref{2.2}, either $A$ has a unitary part or $A$ is c.n.u. and $W(B)$ is a circular disc centered at the origin. In the latter case, Lemma \ref{2.7} yields the unitary similarity of $B$ and $J_m$, and then Theorem \ref{2.5} gives $\|A^{m-1}\|=1$. The converse also follows from Theorem \ref{2.5}.
\end{proof}

\vspace{3mm}

Note that, under the conditions of Proposition \ref{2.6}, if $A$ is c.n.u., then we automatically have $m\le n$. This is because if, otherwise, $m>n$, then $\|A^{m-1}\|=1$ yields, by Proposition \ref{1.2} (a) and (c), that $A$ has a unitary part.

\vspace{3mm}

A specific example of the results obtained so far is in the next proposition.

\vspace{3mm}

\begin{proposition}\label{2.8}
Let $n$ and $m$ be positive integers. Then $W(J_n\otimes J_m)=W(J_{\ell})$, where $\ell=\min\{n, m\}$, and thus $w(J_n\otimes J_m)=\min\{w(J_n), w(J_m)\}$.
\end{proposition}

\vspace{1pt}

\begin{proof}
Assume that $m\le n$. Since the principal submatrix of $J_n\otimes J_m$ formed by its rows and columns numbered $1, m+2, 2m+3, \ldots$, and $(m-1)m+m$ is $J_m$, we have that $J_n\otimes J_m$ is a dilation of $J_m$. Thus $w(J_m)\le w(J_n\otimes J_m)$. The reversed inequality $w(J_n\otimes J_m)\le\|J_n\|w(J_m)=w(J_m)$ is by Proposition \ref{1.1}. Therefore, $w(J_n\otimes J_m)=w(J_m)$ holds. As was seen in the proof of (c) $\Rightarrow$ (a) in Theorem \ref{2.5}, $W(J_n\otimes J_m)$ is a circular disc centered at the origin. Thus the equality of $w(J_n\otimes J_m)$ and $w(J_m)$ implies that of $W(J_n\otimes J_m)$ and $W(J_m)$.
\end{proof}

\vspace{3mm}

Besides $S_n$-matrices, another generalization of the Jordan blocks is the companion matrices. Recall that a companion matrix is one of the form
\[\left[ \begin{array}{ccccccc}
0 & 1 &  &  &  &  &  \\
  & 0 & 1&  &  &  &  \\
  &   & \cdot & \cdot &  &  &  \\
  &   &  & \cdot & \cdot &  &  \\
  &   &   &  & \cdot & \cdot &  \\
  &   &  &  &  & 0 & 1 \\
-a_n & -a_{n-1} & \cdot & \cdot & \cdot & -a_2 & -a_1
\end{array} \right],\]
whose characteristic and minimal polynomials are both equal to $z^n+\sum_{j=1}^na_jz^{n-j}$. The numerical ranges of such matrices have been studied in \cite{5,6,1}.

\vspace{3mm}

\begin{proposition}\label{2.9}
Let $A$ be an $n$-by-$n$ $(n\ge 2)$ companion matrix. Then the following conditions are equivalent:
\begin{enumerate}
\item[\rm (a)] $w(A\otimes A)=\|A\|w(A)$,
\item[\rm (b)] $A$ is unitary, $A=J_n$, or $A$ is unitarily similar to a direct sum $[a\omega_n^j]\oplus B$, where $|a|>1$, $\omega_n=e^{i(2\pi/n)}$, $0\le j\le n-1$, and $B$ is an $S_{n-1}$-matrix with eigenvalues $(1/\overline{a})\omega_n^k$, $0\le k\le n-1$ and $k\neq j$, and
\item[\rm (c)] $p_A=n_A=\infty$ or $n-1$.
\end{enumerate}
\end{proposition}

\vspace{1pt}

\begin{proof}
To prove (a) $\Rightarrow$ (b), let $A'=A/\|A\|$. Then (a) gives $w(A'\otimes A')=w(A')$. By Theorem \ref{2.2}, either $A'$ has a unitary part or it is c.n.u. with numerical range a circular disc centered at the origin. In the former case, either $A$ is normal or is unitarily similar to a matrix of the form $[a\omega_n^j]\oplus B$, where $|a|=\|A\|\ge 1$ and $B$ is of size $n-1$ with eigenvalues $(1/\overline{a})\omega_n^k$, $0\le k\le n-1$ and $k\neq j$ (cf. \cite[Theorem 1.1 and Corollary 1.3]{5}). If $A$ is normal or $|a|=1$, then $A$ is unitary by \cite[Corollary 1.2]{5}. Hence we may assume that $|a|>1$. Thus the eigenvalues of $B$ are all contained in $\mathbb{D}$. Moreover, by \cite[Theorem 2.1]{1}, we have $\rank(I_{n-1}-B^*B)=1$. These two together imply, by way of the singular value decomposition of $B$, that $\|B\|=1$. Hence $B$ is of class $S_{n-1}$. On the other hand, if it is the latter case, then $W(A)$ is also a circular disc centered at the origin. Therefore, $A=J_n$ by \cite[Theorem 2.9]{5}. This proves (b).

\vspace{3mm}

For (b) $\Rightarrow$ (c), if $A$ is unitary (resp., $A=J_n$), then, obviously, $p_A=n_A=\infty$ (resp., $p_A=n_A=n-1$). On the other hand, if $A$ is unitarily similar to the asserted $[a\omega_n^j]\oplus B$, then $\|A\|=\max\{|a|, \|B\|\}=|a|=\rho(A)$. Thus $p_A=n_A=\infty$ by Proposition \ref{1.2} (c) and \ref{1.3}.

\vspace{3mm}

Finally, for (c) $\Rightarrow$ (a), if $p_A=n_A=\infty$, then (a) is a  consequence of Lemma \ref{2.1}. On the other hand, if $p_A=n_A=n-1$, then $A^n=0_n$. This implies that $A=J_n$ and thus (a) holds by Proposition \ref{2.8}.
\end{proof}

\vspace{3mm}

The next theorem is a consequence of Theorem \ref{2.5}. It gives a lower bound, in terms of $p_A$, for $w(A)$ when $A$ is an $n$-by-$n$ matrix with $\|A\|=1$.

\vspace{3mm}

\begin{theorem}\label{2.10}
If $A$ is an $n$-by-$n$ matrix with $\|A\|=\|A^k\|=1$ for some $k\ge 1$, then $w(A)\ge\cos(\pi/(k+2))$. Moreover, in this case, the following conditions are equivalent:
\begin{enumerate}
\item[\rm (a)] $w(A)=\cos(\pi/(k+2))$,
\item[\rm (b)] $A$ is unitarily similar to $J_{k+1}\oplus B$, where $B$ is a finite matrix with $w(B)\le\cos(\pi/(k+2))$, and
\item[\rm (c)] $W(A)=\{z\in\mathbb{C}: |z|\le\cos(\pi/(k+2))\}$.
\end{enumerate}
\end{theorem}

\vspace{3mm}

For the proof of (a) $\Rightarrow$ (b), we need the following lemma.

\vspace{3mm}

\begin{lemma}\label{2.11}
Let
\[ A=\left[
    \begin{array}{ccccc}
      0 \ & a_1 &   &   &   \\
        & 0 & \ \ddots &   &   \\
        &   & \ddots \ & a_{n-2} &   \\
        &   &   & 0 & a_{n-1} \\
        &   &   &   & a \\
    \end{array}
  \right] \ \ \ and \ \ \ B=\left[
                            \begin{array}{cccc}
                              0 \ & a_1 &   &   \\
                                & 0 & \ \ddots &   \\
                                &   & \ddots \ & a_{n-2} \\
                                &   &   & 0 \\
                            \end{array}
                          \right] \]
be $n$-by-$n$ and $(n-1)$-by-$(n-1)$ matrices, respectively, where $n\ge 2$ and $a_j$ is nonzero for all $j$. Then $w(A)>w(B)$.
\end{lemma}

\vspace{1pt}

\begin{proof}
We prove this by induction on $n$. If $n=2$, then $A=${\footnotesize$\left[ \begin{array}{cc} 0 & a_1 \\ 0 & a \\ \end{array} \right]$} and $B=[0]$, in which case we obviously have $w(A)>0=w(B)$. Assume now that the assertion is true for the matrix $A$ of size at most $n-1$ ($n\ge 3$), and let $A$ and $B$ be of the above form. By considering $e^{i\theta}A$ for a suitable real $\theta$ instead of $A$, we may assume that $w(A)$ equals the largest eigenvalue of $\re A$. Let
\[ C=\left[ \begin{array}{cccc}
                              0 \ & a_1 &   &   \\
                                & 0 & \ \ddots &   \\
                                &   & \ddots \ & a_{n-3} \\
                                &   &   & 0 \\
                            \end{array}
                          \right], \]
and let $p(z)$, $q(z)$ and $r(z)$ be the characteristic polynomials of $\re A$, $\re B$ and $\re C$, respectively. We expand the determinant of
\[ \left[
    \begin{array}{ccccc}
      z  & -a_1/2 &   &   &   \\
      -\overline{a}_1/2  & z &  \ddots &   &   \\
        & \ddots  & \ddots  & \ddots &   \\
        &   & \ddots  & z & -a_{n-1}/2 \\
        &   &   & -\overline{a}_{n-1}/2  & z-\re a \\
    \end{array}
  \right]\]
by minors on its last row to obtain $p(z)=(z-\re a)q(z)-(|a_{n-1}|^2/4)r(z)$. Let $\alpha$, $\beta$ and $\gamma$ be the largest eigenvalues of $\re A$, $\re B$ and $\re C$, respectively. Then $\alpha=w(A)$, $\beta=w(B)$ and $\gamma=w(C)$. Since $\re B$ (resp., $\re C$) is a principal submatrix of $\re A$ (resp., $\re B$), we have $\beta\le \alpha$ (resp., $\gamma\le\beta$). Assume that $\alpha=\beta$. Then the above equation yields
\[ 0=p(\alpha)=(\alpha-\re a)q(\beta)-\frac{1}{4}|a_{n-1}|^2\gamma(\beta)=-\frac{1}{4}|a_{n-1}|^2\gamma(\beta).\]
Since $a_{n-1}\neq 0$ and $\beta$ is larger than or equal to all eigenvalues of $\re C$, we infer from $\gamma(\beta)=0$ that $\beta=\gamma$ or $w(B)=w(C)$. This contradicts our induction hypothesis for $B$ and $C$. Hence we must have $\alpha>\beta$ or $w(A)>w(B)$.
\end{proof}

\vspace{2mm}

\begin{proof}[Proof of Theorem $\ref{2.10}$]
By Theorem \ref{2.5}, the assumption $\|A\|=\|A^k\|=1$ implies that $w(A\otimes J_{k+1})=w(J_{k+1})$. Hence
\[ w(A)=\|J_{k+1}\|w(A)\ge w(A\otimes J_{k+1})=w(J_{k+1})=\cos\frac{\pi}{k+2} \]
as asserted.

\vspace{3mm}

We now prove the equivalence of (a), (b) and (c). The implications (b) $\Rightarrow$ (c) and (c) $\Rightarrow$ (a) are trivial. To prove (a) $\Rightarrow$ (b), let $x$ be a unit vector in $\mathbb{C}^n$ such that $\|A^kx\|=1$. Then $\|A^jx\|=1$ for all $j$, $0\le j\le k$. We now check that $A^{k+1}x=0$. Assuming otherwise that $\|A^{k+1}x\|>0$, let $u_t=[u_{t 1} \ \ldots \ u_{t, k+2}]^T$ in $\mathbb{C}^{k+2}\otimes \mathbb{C}^n$, where
\[ u_{t j}=\left\{ \begin{array}{ll} \frac{\textstyle\sqrt{1-t^2}}{\textstyle\|A^{k+1}x\|} A^{k+1}x & \ \ \ \mbox{if } \ j=1, \vspace{3mm} \\
t\sqrt{\frac{\textstyle 2}{\textstyle k+2}} \sin\frac{\textstyle (j-1)\pi}{\textstyle k+2} A^{k-j+2}x & \ \ \ \mbox{if } \ j=2, \ldots, k+2
\end{array}\right.\]
for any $t$, $0<t<1$. Note that
\[v\equiv \sqrt{\frac{2}{k+2}}\left[\sin\frac{\pi}{k+2} \ \  \sin\frac{2\pi}{k+2} \ \  \ldots \ \  \sin\frac{(k+1)\pi}{k+2}\right]^T \]
is a unit vector in $\mathbb{C}^{k+1}$ with $\langle J_{k+1}v, v\rangle=\cos(\pi/(k+2))$ (cf. \cite[Proposition 1 (3)]{8}). Hence $\|u_t\|=((1-t^2)+t^2\|v\|^2)^{1/2}=1$, and
\begin{eqnarray*}
\langle (J_{k+2}\otimes A)u_t, u_t\rangle &= & t\sqrt{1-t^2}\sqrt{\frac{2}{k+2}}\sin\frac{\pi}{k+2}\|A^{k+1}x\|\\
&& +t^2\frac{2}{k+2}\sum_{j=1}^k\sin\frac{j\pi}{k+2}\sin\frac{(j+1)\pi}{k+2}\|A^{k-j+1}x\|^2\\
&= & t\sqrt{1-t^2}\sqrt{\frac{2}{k+2}}\sin\frac{\pi}{k+2}\|A^{k+1}x\| + t^2\langle J_{k+1}v, v\rangle\\
&= & t\sqrt{1-t^2}\sqrt{\frac{2}{k+2}}\sin\frac{\pi}{k+2}\|A^{k+1}x\| + t^2\cos\frac{\pi}{k+2}.
\end{eqnarray*}
To reach a contradiction, we need to find some $t_0$, $0<t_0<1$, such that $\langle (J_{k+2}\otimes A)u_{t_0}, u_{t_0}\rangle > \cos(\pi/(k+2))$. This is the same as
\[ t_0\sqrt{1-t_0^2}\sqrt{\frac{2}{k+2}}\sin\frac{\pi}{k+2}\|A^{k+1}x\| > (1-t_0^2)\cos\frac{\pi}{k+2} \]
or
\[ \frac{t_0}{\sqrt{1-t_0^2}} > \sqrt{\frac{k+2}{2}}\frac{\cot\frac{\pi}{k+2}}{\|A^{k+1}x\|}.\]
Since $\lim_{t\rightarrow 1^-}t/\sqrt{1-t^2}=\infty$, the existence of such a $t_0$ is guaranteed. On the other hand, we also have
\[ \langle (J_{k+2}\otimes A)u_{t_0}, u_{t_0}\rangle \le w(J_{k+2}\otimes A)\le\|J_{k+2}\|w(A)=w(A)=\cos\frac{\pi}{k+2},\]
hence a contradiction. Thus we must have $A^{k+1}x=0$. Let $K$ be the subspace of $\mathbb{C}^n$ generated by $x, Ax, \ldots, A^kx$. Then $AK\subseteq K$. If $A'$ is the restriction of $A$ to $K$, then ${A'}^{k+1}=0$ and $\|{A'}^jx\|=\|A^jx\|=1$ for all $j$, $0\le j\le k$. Hence $\|{A'}^j\|=1$ for all such $j$'s. Together with ${A'}^{k+1}=0$, this says that $p_{A'}=k$ and thus $\dim K=k+1$ by Proposition \ref{1.2} (a). Therefore, $A'$ is unitarily similar to a matrix of the form $[a_{ij}]_{i,j=1}^{k+1}$ with $a_{ij}=0$ for all $i\ge j$. Since $1=\|{A'}^k\|=|a_{12}\cdots a_{k, k+1}|$, we infer that $|a_{12}|=\cdots= |a_{k, k+1}|=1$, and thus all the other $a_{ij}$'s are zero. Therefore, $[a_{ij}]_{i,j=1}^{k+1}$, and hence $A'$, is unitarily similar to $J_{k+1}$. Then $A$ is unitarily similar to a matrix of the form
\[\left[\begin{array}{c|c}  \ \ J_{k+1} \ \ & \begin{array}{c}  0 \\ \, b_1 \ \cdots \ b_{n-k-1}\end{array} \\ \hline  0 & \begin{array}{ccc}  c_1 & & *\\ & \ddots & \\ * & & c_{n-k-1}\end{array}\end{array}\right].\]
To show that all the $b_j$'s are zero, we appeal to Lemma \ref{2.11}. Indeed, for each $j$, $1\le j\le n-k-1$, consider the $(k+2)$-by-$(k+2)$ matrix
\[ A_j=\left[\begin{array}{c|c}  \ \ J_{k+1} \ \ & \begin{array}{c}  0 \\ \vdots \\ 0 \\ b_j \end{array} \\ \hline  0 & c_j \end{array}\right].\]
If $b_j\neq 0$, then $w(A_j)>w(J_{k+1})=\cos(\pi/(k+2))$ by Lemma \ref{2.11}, which contradicts $w(A_j)\le w(A)=\cos(\pi/(k+2))$. This proves (a) $\Rightarrow$ (b).
\end{proof}

\vspace{3mm}

Theorem \ref{2.10} generalizes the classical result of Williams and Crimmins \cite{17} for $k=1$. The following corollary is for $k=n-1$. Part of it has been proven in \cite{19}: the equivalence of (b) and (c) is in \cite[Theorem 1]{19} and that of (b) and (d) in \cite[p. 352]{19}.

\vspace{3mm}

\begin{corollary}\label{2.12}
The following conditions are equivalent for an $n$-by-$n$ matrix $A$ with $\|A\|=1$:
\begin{enumerate}
\item[\rm (a)] $\|A^{n-1}\|=1$ and $w(A)=\cos(\pi/(n+1))$,
\item[\rm (b)] $A$ is unitarily similar to $J_{n}$,
\item[\rm (c)] $W(A)=\{z\in\mathbb{C}: |z|\le\cos(\pi/(n+1))\}$,
\item[\rm (d)] $\|A^{n-1}\|=1$ and $A^n=0_n$, and
\item[\rm (e)] $p_A=n_A=n-1$.
\end{enumerate}
\end{corollary}

\vspace{1pt}

\begin{proof}
The equivalence of (a) and (b) is by Theorem \ref{2.10}. The other implications are either in \cite{19} or trivial.
\end{proof}

\vspace{3mm}

Note that, in the preceding corollary, the conditions that $\|A\|=1$ and $w(A)=\cos(\pi/(n+1))$ for an $n$-by-$n$ matrix $A$ are not sufficient to guarantee that $A$ be unitarily similar to $J_n$. One example is $A=J_{n-1}\oplus[\cos(\pi/(n+1))]$.

\vspace{3mm}

We end this section with a characterization of matrices $A$ satisfying $p_A=n_A$. This is related to the previous results.

\vspace{3mm}

\begin{theorem}\label{2.13}
Let $A$ be an $n$-by-$n$ matrix with $\|A\|=1$. Then
\begin{enumerate}
\item[\rm (a)] $A$ satisfies $p_A=n_A\ (\le\infty)$ if and only if either it has a unitary part or is unitarily similar to a direct sum $J_{k+1}\oplus B$, where $k=p_A<\infty$ and $B^{k+1}=0_{n-k-1}$, and
\item[\rm (b)] if $p_A=n_A\ (\le \infty)$, then $w(A\otimes A)=w(A)$ holds, but not conversely.
\end{enumerate}
\end{theorem}

\vspace{1pt}

\begin{proof}
(a) For the necessity, we may assume, in view of Proposition \ref{1.2} (c), that $k\equiv p_A=n_A<\infty$ and prove that $A$ is unitarily similar to the asserted direct sum. Since $A^{k+1}=0_n$, $A$ is unitarily similar to a block matrix $A'$ of the form $[A_{ij}]_{i,j=1}^{k+1}$ with $A_{ij}=0$ for $1\le j\le i\le k+1$. Hence
\[ {A'}^k=\left[
            \begin{array}{cccc}
              0 & \cdots & 0 & \prod_{i=1}^kA_{i, i+1} \\
                & 0 &   & 0 \\
                &   & \ddots  & \vdots \\
                &   &   & 0 \\
            \end{array}
          \right]. \]
Since $\|{A'}^k\|=\|A^k\|=\|A\|^k=1$, we have $\|\prod_{i=1}^kA_{i, i+1}\|=1$. Let $x$ be a unit vector such that $\|(\prod_{i=1}^kA_{i, i+1})x\|=1$. Then $\|(\prod_{i=j}^kA_{i, i+1})x\|=1$ for all $j$, $1\le j\le k$. Let $\{e_1, \ldots, e_{k+1}\}$ be the standard basis for $\mathbb{C}^{k+1}$, and let $x_j=e_j\otimes(\prod_{i=j}^kA_{i, i+1})x$ if $1\le j\le k$, and $x_{k+1}=e_{k+1}\otimes x$. Then $x_1, \ldots, x_{k+1}$ are orthonormal vectors in $\mathbb{C}^n$, and $A'x_1=0$ and $A'x_j=x_{j-1}$ for $2\le j\le k+1$. Thus if $K$ is the subspace generated by $x_1, \ldots, x_{k+1}$, then $\dim K=k+1$, $A'K\subseteq K$, and the restriction of $A'$ to $K$ is unitarily similar to $J_{k+1}$. We infer from $\|A'\|=1$ and ${A'}^*x_{k+1}=0$ that $K$ reduces $A'$, and thus $A'$ is unitarily similar to $J_{k+1}\oplus B$ with $B^{k+1}=0$.

\vspace{3mm}

For the converse, if $A$ has a unitary part, then $p_A=n_A=\infty$ by Proposition \ref{1.2} (c). On the other hand, if $A$ is unitarily similar to $J_{k+1}\oplus B$ with the asserted properties, then $A^{k+1}=0$ implies that $p_A\le n_A\le k$. But
\[ \|A^k\|=\|J_{k+1}^k\oplus B^k\|=\max\{\|J_{k+1}^k\|, \|B^k\|\}=1=\|A\|^k \]
and $\|A^{k+1}\|=0<1=\|A\|^{k+1}$ together yield $p_A=n_A=k$.

\vspace{3mm}

(b) If $A$ has a unitary part, then $w(A\otimes A)=w(A)$ by Proposition \ref{2.1}. On the other hand, if $A$ is unitarily similar to $J_{k+1}\oplus B$ as in (a), then $A\otimes A$ is unitarily similar to $(J_{k+1}\otimes J_{k+1})\oplus (J_{k+1}\otimes B)\oplus(B\otimes J_{k+1})\oplus(B\otimes B)$. Note that $w(J_{k+1}\otimes J_{k+1})=w(J_{k+1})$ by Proposition \ref{2.8}, and
\begin{equation}\label{eq1}
w(J_{k+1}\otimes B)=w(B\otimes J_{k+1})\le\|J_{k+1}\|w(B)=w(B)
\end{equation}
by Proposition \ref{1.1}. Since $B^{k+1}=0$ and $\|B\|\le 1$, \cite[Lemma 3 (a)]{21} implies that $B$ can be dilated to the direct sum of $\rank(I-B^*B)$ copies of $J_m$ for some $m\le k+1$. Thus $w(B)\le w(J_m)\le w(J_{k+1})$. Combined with \eqref{eq1}, this yields $w(J_{k+1}\otimes B)\le w(J_{k+1})$. Also,
\[w(B\otimes B)\le\|B\|w(B)\le w(B)\le w(J_{k+1}).\]
Therefore,
\begin{align*}
w(A\otimes B) &= \max\{w(J_{k+1}\otimes J_{k+1}), w(J_{k+1}\otimes B), w(B\otimes B)\}\\
&= w(J_{k+1})\\
&= \max\{w(J_{k+1}), w(B)\}\\
&= w(A).
\end{align*}

\vspace{1mm}

That $w(A\otimes A)=w(A)$ does not imply $p_A=n_A$ is seen by $A=J_2\oplus[a]$, where $0<|a|\le 1/2$, in which case, $\|A\|=1$ and $w(A\otimes A)=w(A)=1/2$, but $p_A=1$ and $n_A=\infty$.
\end{proof}

\vspace{3mm}

The final result of this section is conditions for a matrix $A$ with $p_A=n_A$ so that it be unitarily similar to a block-shift matrix
\begin{equation}\label{eq2}
A'=\left[
    \begin{array}{cccc}
      0 & A_1 &   &   \\
        & 0 & \ddots &   \\
        &   & \ddots & A_k \\
        &   &   & 0 \\
    \end{array}
  \right]
\end{equation}
with $\|A_1\cdots A_k\|=\|A\|$.

\vspace{3mm}

\begin{proposition}\label{2.14}
Let $A$ be an $n$-by-$n$ matrix with $p_A=n_A\equiv k<\infty$. If either {\rm (a)} $k=1$, $n-2$ or $n-1$, or {\rm (b)} $n=2, 3, 4$ or $5$, then $A$ is unitarily similar to the block-shift matrix $A'$ in \eqref{eq2} with $\|A_1\cdots A_k\|=\|A\|$.
\end{proposition}

\vspace{1pt}

\begin{proof}
We may assume that $\|A\|=1$.

\vspace{3mm}

(a) If $k=n_A=1$, then $A^2=0_n$. Hence $A$ is unitarily similar to a block-shift matrix of the form {\footnotesize $\left[
    \begin{array}{cc}
      0 & A_1  \\
      0  & 0
    \end{array}
  \right]$} with $\|A_1\|=\|A\|$.

\vspace{3mm}

If $k=p_A=n_A=n-1$ (resp., $n-2$), then Theorem \ref{2.13} (a) implies that $A$ is unitarily similar to $J_n$ (resp., $J_{n-1}\oplus[0]$). The latter matrix plays the role of $A'$ with $k=n-1$ (resp., $n-2$) and $A_1=\cdots=A_{n-1}=[1]$ (resp., $A_1=\cdots=A_{n-3}=[1]$ and $A_{n-2}=[1 \ \ 0]$).

\vspace{3mm}

(b) In light of (a), we need only prove for $n=5$ and $k=2$. Invoking Theorem \ref{2.13} to obtain the unitary similarity of $A$ and $J_3\oplus${\footnotesize $\left[
    \begin{array}{cc}
      0 & b  \\
      0  & 0
    \end{array}
  \right]$}, where $|b|\le 1$. The latter matrix is permutationally similar to a block-shift matrix $A'$ with $k=2$, $A_1=${\footnotesize $\left[
    \begin{array}{cc}
      1 & 0  \\
      0  & b
    \end{array}
  \right]$} and $A_2=${\footnotesize $\left[
    \begin{array}{c}
       1  \\
      0
    \end{array}
  \right]$}. We obviously have $\|A_1A_2\|=\|${\footnotesize $\left[
    \begin{array}{c}
       1  \\
      0
    \end{array}
  \right]$}$\|=1=\|A\|$.
\end{proof}

\vspace{3mm}

We remark that the preceding proposition fails for $n=6$ and $k=2$. Here is an example. Let $A=J_3\oplus B$, where
\[B=b \left[
    \begin{array}{ccc}
      0 & 1 & 1  \\
      0  & 0 & 1 \\
      0 & 0 & 0
    \end{array}
  \right]\]
with $b=\sqrt{2/(3+\sqrt{5})}$. Then $\|A^2\|=1=\|A\|^2$ and $A^3=0_6$. This shows that $p_A=n_A=2$. Since $w(B)=2b>\sqrt{2}/2=w(J_3)$ and $w(B)$ is not a circular disc centered at the origin (cf. \cite[Theorem 4.1 (2)]{13}), we infer that nor is $W(A)$ ($=$ the convex hull of $W(J_3)\cup W(B)$). This implies that $A$ cannot be unitarily similar to a block-shift matrix.

\vspace{10mm}

\section{Nonnegative Matrices}

Recall that a matrix $A=[a_{ij}]_{i,j=1}^n$ is \emph{nonnegative} (resp., \emph{positive}), denoted by $A\succcurlyeq 0$ (resp., $A\succ 0$), if $a_{ij}\ge 0$ (resp., $a_{ij}>0$) for all $i$ and $j$. Two $n$-by-$n$ matrices $A$ and $B$ are \emph{permutationally similar} if there is an $n$-by-$n$ \emph{permutation matrix} $P$ (one with each row and column has exactly one 1 and all other entries 0) such that $P^TAP=B$. $A$ is said to be (\emph{permutationally}) \emph{reducible} if either $A$ is the 1-by-1 zero matrix or $n\ge 2$ and it is permutationally similar to a matrix of the form {\footnotesize$\left[ \begin{array}{cc} B & C \\ 0 & D \\ \end{array} \right]$}, where $B$ and $D$ are square matrices; otherwise, it is (\emph{permutationally}) \emph{irreducible}. It is known that if $A$ is nonnegative with $\re A$ irreducible, then it is permutationally similar to a block-shift matrix if and only if its numerical range is a circular disc centered at the origin (cf. \cite[Theorem 1 (a)$\Leftrightarrow$(r)]{16}). Other properties of nonnegative matrices can be found in \cite[Section 6.2 and Chapter 8]{11}.

\vspace{3mm}

The main result of this section is the following theorem, which essentially generalizes Theorem \ref{2.5}.

\vspace{3mm}

\begin{theorem}\label{3.1}
Let $A$ be an $n$-by-$n$ matrix and $B$ an $m$-by-$m$ nonnegative matrix with $\re B$ irreducible. Then the following conditions are equivalent:
\begin{enumerate}
\item[\rm (a)] $w(A\otimes B)=\|A\|w(B)$,
\item[\rm (b)] either $p_A=\infty$ or $n_B\le p_A<\infty$ and $W(B)$ is a circular disc centered at the origin, and
\item[\rm (c)] either $p_A=\infty$ or $n_B\le p_A<\infty$ and $B$ is permutationally similar to a block-shift matrix
\[\left[
    \begin{array}{cccc}
      0 & B_1 &   &   \\
        & 0 & \ddots &   \\
        &   & \ddots & B_k \\
        &   &   & 0 \\
    \end{array}
  \right]\]
with $k=n_B$.
\end{enumerate}
\end{theorem}

\vspace{3mm}

For its proof, we need the following two lemmas.

\vspace{3mm}

\begin{lemma}\label{3.2}
Let $A=[a_{ij}]_{i,j=1}^n$ be a nonnegative matrix. Then the following hold:
\begin{enumerate}
\item[\rm (a)] The index $n_A$ is finite if and only if there is no sequence of indices $i_0, i_1, \ldots, i_{k-1}, i_k\ (k\ge 1)$ with $i_0=i_k$ such that $a_{i_0 i_1}, \ldots, a_{i_{k-1} i_k}$ are all nonzero. In particular, we have $n_A=\sup\{k\ge 1 : \mbox{there are distinct } i_j, 0\le j\le k, \mbox{such that } a_{i_j i_{j+1}}\neq 0 \mbox{ for all } j\}$.
\item[\rm (b)] $n_A=\infty$ if and only if there is a $k\ge 1$ such that some diagonal entry of $A^k$ is nonzero.
\item[\rm (c)] If $a_{ii}\neq 0$ for some $i$, $1\le i\le n$, then $n_A=\infty$.
\item[\rm (d)] If $A$ is irreducible, then $n_A=\infty$.
\item[\rm (e)] If $A$ is the block-shift matrix
\[\left[
    \begin{array}{cccc}
      0_{n_1} & A_1 &   &   \\
        & 0_{n_2} & \ddots &   \\
        &   & \ddots & A_k \\
        &   &   & 0_{n_{k+1}} \\
    \end{array}
  \right] \ \ \mbox{on}  \ \ \mathbb{C}^n=\mathbb{C}^{n_1}\oplus\cdots\oplus \mathbb{C}^{n_{k+1}}\]
and $\re A$ is irreducible, then $k=n_A$.
\end{enumerate}
\end{lemma}

\vspace{1mm}

\begin{proof}
(a) Assume first that the indices $i_0, i_1, \ldots, i_{k-1}, i_k=i_0$ ($k\ge 1$) are such that $a_{i_0 i_1},\ldots, a_{i_{k-1} i_k}\neq 0$. \cite[Theorem 6.2.16]{11} says that this is the case if and only if $(A^k)_{i_0 i_0}$, the $(i_0, i_0)$-entry of $A^k$, is nonzero. Hence $A^k\neq 0_n$. Similarly, considering the sequence $i_0, \ldots, i_k, i_1, \ldots, i_k, \ldots, i_1, \ldots, i_k$ of $\ell k+1$ indices for any $\ell\ge 1$, we also obtain $A^{\ell k}\neq 0_n$. It follows that $n_A=\infty$. Conversely, assume that $n_A=\infty$. Then $A^k\neq 0_n$ for some $k\ge n$. \cite[Theorem 6.2.16]{11} yields that, for some $i$ and $j$, there are indices $i_0=i, i_1, \ldots, i_{k-1}, i_k=j$ such that $a_{i_0 i_1},\ldots, a_{i_{k-1} i_k}$ are all nonzero. By the pigeonhole principle, we infer that $i_s=i_t$ for some $s$ and $t$, $0\le s< t\le k$. Then $i_s, \ldots, i_t$ are such that $i_s=i_t$ and $a_{i_s i_{s+1}},\ldots, a_{i_{t-1} i_t}\neq 0$. This proves the converse. The expression for $n_A$ is an easy consequence of \cite[Theorem 6.2.16]{11} and the above arguments. So are (b) and (c).

\vspace{3mm}

(d) Note that the irreducibility of $A$ is equivalent to the existence, for every distinct pair $i$ and $j$, of indices $i_0=i, i_1, \ldots, i_{k-1}, i_k=j$ ($k\ge 1$) such that $a_{i_0 i_1},\ldots, a_{i_{k-1} i_k}$ are all nonzero. Combining such indices from $i$ to $j$ with those from $j$ to $i$ yields one from $i$ to $i$ with the corresponding entries nonzero. Thus $n_A=\infty$ by \cite[Theorem 6.2.16]{11} and (b).

\vspace{3mm}

(e) Since $A^{k+1}=0_n$, we have $n_A\le k$. If $n_A<k$, then $A^k=0_n$, which implies that $A_1\cdots A_k=0$. If there are any nonzero $a_{i_0 i_1}, a_{i_1 i_2}, \ldots, a_{i_{k-1} i_k}$, where $(\sum_{j=1}^{\ell}n_j)+1\le i_{\ell}\le\sum_{j=1}^{\ell+1}n_j$ for $0\le\ell\le k$, then the $(i_0, n_{k+1} - (n-i_{k}))$-entry of $A_1\cdots A_k$, being larger than or equal to $\prod_{j=0}^{k-1}a_{i_j i_{j+1}}$, is nonzero, which contradicts the zeroness of the product $A_1\cdots A_k$. Thus no such nonzero sequence exists. This results in the reducibility of $\re A$, a contradiction. Hence we must have $n_A=k$.
\end{proof}

\vspace{3mm}

We remark that the conditions in the preceding lemma can all be expressed equivalently in terms of the directed graph associated with the matrix $A$ (cf. \cite[Section 6.2]{11}).

\vspace{3mm}

\begin{lemma}\label{3.3}
Let $A$ and $B$ be $n$-by-$n$ and $m$-by-$m$ matrices, respectively. If $B$ is unitarily similar to a block-shift matrix
\begin{equation}\label{eq3}
\left[
    \begin{array}{cccc}
      0_{m_1} & B_1 &   &   \\
        & 0_{m_2} & \ddots &   \\
        &   & \ddots & B_k \\
        &   &   & 0_{m_{k+1}} \\
    \end{array}
  \right] \ \ \mbox{on}  \ \ \mathbb{C}^m=\mathbb{C}^{m_1}\oplus\cdots\oplus \mathbb{C}^{m_{k+1}}
\end{equation}
with $k\le p_A\le\infty$, then $w(A\otimes B)=\|A\|w(B)$.
\end{lemma}

\vspace{1mm}

\begin{proof}
We may assume that $\|A\|=1$ and $B$ is equal to the block-shift matrix \eqref{eq3}. Since $k\le p_A\le\infty$, we have $\|A^k\|=\|A\|^k=1$. Let $x$ be a unit vector in $\mathbb{C}^n$ such that $\|A^kx\|=1$, and let $y=[y_1 \ \ldots \ y_{k+1}]^T$, where $y_j$ is in $\mathbb{C}^{m_j}$, $1\le j\le k+1$, be a unit vector in $\mathbb{C}^m$ such that $|\langle By, y\rangle|=w(B)$. Let $u=[y_1\otimes A^kx \ \ y_2\otimes A^{k-1}x \ \ \ldots \ \ y_{k+1}\otimes x]^T$. Then $u$ is a vector in $\mathbb{C}^m\otimes \mathbb{C}^n$ with
\begin{align*}
\|u\| &= (\sum_{j=1}^{k+1}\|y_j\otimes A^{k-j+1}x\|^2)^{1/2} = (\sum_{j=1}^{k+1}\|y_j\|^2 \|A^{k-j+1}x\|^2)^{1/2} \\
&= (\sum_{j=1}^{k+1}\|y_j\|^2)^{1/2} = \|y\| = 1.
\end{align*}
Moreover, we have
\begin{eqnarray*}
&& |\langle (B\otimes A)u, u\rangle|\\
&=& \left| \left\langle \left[
    \begin{array}{cccc}
      0_{m_1n} & B_1\otimes A &   &   \\
        & 0_{m_2n} & \ddots &   \\
        &   & \ddots & B_k\otimes A \\
        &   &   & 0_{m_{k+1}n} \\
    \end{array}
  \right] \left[
    \begin{array}{c}
      y_1\otimes A^kx   \\
       y_2\otimes A^{k-1}x   \\
        \vdots \\
       y_{k+1}\otimes x
    \end{array}
  \right], \left[
    \begin{array}{c}
      y_1\otimes A^kx   \\
       y_2\otimes A^{k-1}x   \\
        \vdots \\
       y_{k+1}\otimes x
    \end{array}
  \right]\right\rangle \right| \\
&=& | \sum_{j=1}^{k}\langle (B_jy_{j+1})\otimes(A^{k-j+1}x), y_j\otimes(A^{k-j+1}x)\rangle | \\
&=& | \sum_{j=1}^{k}\langle B_jy_{j+1}, y_j\rangle\|A^{k-j+1}x\|^2 | \\
&=& | \sum_{j=1}^{k}\langle B_jy_{j+1}, y_j\rangle | \\
&=& |\langle By, y\rangle| = w(B).
\end{eqnarray*}
This shows that $w(B)\le w(B\otimes A) = w(A\otimes B)$. But $w(A\otimes B)\le \|A\|w(B)=w(B)$ always holds by Proposition \ref{1.1}. Hence $w(A\otimes B)=w(B)$ as asserted.
\end{proof}

\vspace{3mm}

We are now ready to prove Theorem \ref{3.1}.

\vspace{3mm}

\begin{proof}[Proof of Theorem $\ref{3.1}$]
For (a) $\Rightarrow$ (b), We assume that $\|A\|=1$ and $A$ is c.n.u. In view of Theorem \ref{2.2} and Proposition \ref{1.2} (c), we need only check that $w(A\otimes B)=w(B)$ implies $n_B\le p_A\ (<\infty)$. Let $B=[b_{ij}]_{i,j=1}^m$, and let $x$ be a unit vector in $\mathbb{C}^m\otimes \mathbb{C}^n$ such that $w(B\otimes A)=|\langle (B\otimes A)x, x\rangle|$. If $x=[x_1 \ \ldots \ x_m]^T$, where $x_j$ is in $\mathbb{C}^n$ for $1\le j\le m$, then
\begin{align}
w(B) &= w(B\otimes A) = |\langle [b_{ij}A]x, x\rangle| \nonumber \\
&\le \sum_{i,j}b_{ij}|\langle Ax_j, x_i\rangle| \nonumber \\
&\le \sum_{i,j} b_{ij}\|Ax_j\|\|x_i\| \label{eq4}\\
&\le \|A\|\sum_{i,j} b_{ij}\|x_j\|\|x_i\| \label{eq5} \\
&\le \langle Bx', x'\rangle \nonumber \\
&\le w(B), \label{eq6}
\end{align}
where $x'=[\|x_1\| \ \ldots \ \|x_{m}\|]^T$ is a unit vector in $\mathbb{C}^m$. This shows that the above inequalities are equalities throughout. Since $B\succcurlyeq 0$ and $\re B$ is irreducible, there is a unique unit vector $y$ in $\mathbb{C}^m$ with $y\succ 0$ such that $\langle By, y\rangle=w(B)$ (cf. \cite[Proposition 3.3]{14}). The equality in \eqref{eq6} yields that $x'=y$ and thus $x_j\neq 0$ for all $j$. Also, the equalities in \eqref{eq4} and \eqref{eq5} imply that $|\langle Ax_{j}, x_i\rangle|=\|Ax_j\|\|x_i\|=\|x_j\|\|x_i\|$ for all those $b_{ij}$'s with $b_{ij}>0$. Thus $Ax_j=\lambda_{ij}x_i$ for some $\lambda_{ij}$ satisfying $|\lambda_{ij}|=\|x_j\|/\|x_i\|$. Assume first that $k\equiv n_B<\infty$. Thus $B^k\neq 0_m$. By Lemma \ref{3.2} (a), there are distinct indices $i_0, \ldots, i_k$ such that $b_{i_0 i_1}, \ldots, b_{i_{k-1} i_k} >0$. It thus follows from above that $Ax_{i_j}=\lambda_{i_{j-1} i_j}x_{i_{j-1}}$ for $1\le j\le k$. Hence $A^kx_{i_k}=(\prod_{j=1}^k\lambda_{i_{j-1} i_j})x_{i_0}$. Since
\[ \|A^kx_{i_k}\| = (\prod_{j=1}^k\frac{\|x_{i_j}\|}{\|x_{i_{j-1}}\|})\|x_{i_0}\|=\|x_{i_k}\|,\]
we obtain $\|A^k\|=1$ or $p_A\ge k=n_B$. On the other hand, if $n_B=\infty$, then the same arguments as above with $k$ arbitrarily large yield that $p_A=\infty$, which contradicts our assumption that $A$ is c.n.u. This proves (a) $\Rightarrow$ (b).

\vspace{3mm}

That (b) $\Leftrightarrow$ (c) is a consequence of \cite[Theorem 1 (a)$\Leftrightarrow$(r)]{16}, and (c) $\Rightarrow$ (a) is by Lemma \ref{3.2} (e) and Lemma \ref{3.3}.
\end{proof}

\vspace{3mm}

Note that, in Theorem \ref{3.1}, the implication (a) $\Rightarrow$ (b) or (a) $\Rightarrow$ (c) is no longer true if $B$ is nonnegative but without the irreducibility of $\re B$. One example is $A=B=J_2\oplus[a]$, where $0<a\le 1/2$ (cf. the end of the proof of Theorem \ref{2.13} (b)). The next example shows that the same can be said if $B$ is not nonnegative but $\re B$ is irreducible.

\vspace{3mm}

\begin{example}\label{3.4}
Let $A=J_3$ and
\[B = \left[
    \begin{array}{ccc}
      0 & -\sqrt{2} &  1 \\
      0  & 0 & 1   \\
      0  & 0  & \sqrt{2}/2
    \end{array}
  \right].\]
Then $\re B$ is easily seen to be irreducible. We now show that $W(B)=\overline{\mathbb{D}}$. This is seen via \cite[Corollary 2.5]{13} by letting $u=0$ and $\lambda=\sqrt{2}/2$ therein and checking that
\[ \tr(B^*B^2)=\tr\left[
    \begin{array}{ccc}
      0 & 0 &  0 \\
      0  & 0 & 1   \\
      0  & 0  & \sqrt{2}/4
    \end{array}
  \right]=\frac{\sqrt{2}}{4}=\lambda|\lambda|^2\]
and $\tr(B^*B)=9/2\ge 5|\lambda|^2$, where $\tr(\cdot)$ denotes the trace of a matrix. We next prove that 1 is an eigenvalue of $\re(A\otimes B)$. Indeed, since
\[ \re(A\otimes B)=\frac{1}{2}\left[
    \begin{array}{ccc}
      0_3 & B &  0_3 \\
      B^*  & 0_3 & B   \\
      0_3  & B^*  & 0_3
    \end{array}
  \right],\]
we need to check that
\[ \det\left[
    \begin{array}{ccc}
      2I_3 & -B &  0_3 \\
      -B^*  & 2I_3 & -B   \\
      0_3  & -B^*  & 2I_3
    \end{array}
  \right]=0.\]
By a repeated use of the Schur decomposition, the above determinant is seen to be equal to
\begin{eqnarray*}
&& \det\,(2I_3) \, \det\left( \left[
    \begin{array}{cc}
      2I_3 & -B   \\
      -B^*  & 2I_3
    \end{array}
  \right] - \left[\begin{array}{c} -B^* \\ 0_3 \end{array} \right](\frac{1}{2}I_3)\left[-B \ \ 0_3\right]\right)\\
&=& 8 \det \left[
    \begin{array}{cc}
      2I_3-(1/2)B^*B & -B   \\
      -B^*  & 2I_3
    \end{array}
  \right]\\
&=& 8 \det \, (4I_3-B^*B-BB^*)\\
&=& 8 \det \left[
    \begin{array}{ccc}
      1 & -1 &  -\sqrt{2}/2 \\
      -1  & 1 & \sqrt{2}/2   \\
      -\sqrt{2}/2  & \sqrt{2}/2  & 1
    \end{array}
  \right]\\
&=& 0
\end{eqnarray*}
as required. Since $W(A\otimes B)$ is a circular disc centered at the origin (by the unitary similarity of $A\otimes B$ and $e^{i\theta}(A\otimes B)$ for all real $\theta$) and $w(A\otimes B)\le\|A\|w(B)=1$, we infer from $1\in\sigma(\re(A\otimes B))$ that $W(A\otimes B)=\overline{\mathbb{D}}$. Hence $w(A\otimes B)=1=\|A\|w(B)$. But, obviously, we have $n_B=\infty$ and $p_A=2$. \hfill $\square$
\end{example}

\vspace{3mm}

The next corollary gives a more concrete equivalent condition, in terms of block-shift matrices, for $w(A\otimes B)=\|A\|w(B)$ when $A=B\succcurlyeq 0$ and $\re B$ is irreducible.

\vspace{3mm}

\begin{corollary}\label{3.5}
Let $A$ be an $n$-by-$n$ nonnegative matrix with $\re A$ irreducible. Then the following conditions are equivalent:
\begin{enumerate}
\item[\rm (a)] $w(A\otimes A)=\|A\|w(A)$,
\item[\rm (b)] $p_A=n_A\ (\le\infty)$, and
\item[\rm (c)] either $A$ is unitarily similar to $[a]\oplus A'$ with $|a|\ge\|A'\|$, or $A$ is permutationally similar to a block-shift matrix
\[A''=\left[
    \begin{array}{cccc}
      0 & A_1 &   &   \\
        & 0 & \ddots &   \\
        &   & \ddots & A_k \\
        &   &   & 0 \\
    \end{array}
  \right] \]
with $\|A_1\cdots A_k\|=\|A\|$.
\end{enumerate}
\end{corollary}

\vspace{1mm}

\begin{proof}
We may assume that $\|A\|=1$. The implication (a) $\Rightarrow$ (b) is by Theorem \ref{3.1} and Proposition \ref{1.3} (d). For (b) $\Rightarrow$ (c), if $p_A=n_A=\infty$, then $A$ has a unitary part by Proposition \ref{1.2} (c), and hence $A$ is unitarily similar to $[a]\oplus A'$ with $|a|=1\ge\|A'\|$ as asserted. On the other hand, if $p_A=n_A<\infty$, then $w(A\otimes A)=w(A)$ by Theorem \ref{2.13} (b). Hence Theorem \ref{2.2} implies that $W(A)$ is a circular disc centered at the origin. For a nonnegative $A$ with $\re A$ irreducible, this is equivalent to $A$ being permutationally similar to the block-shift matrix $A''$ (cf. \cite[Theorem 1 (a)$\Leftrightarrow$(r)]{16}). As $n_{A''}=k$ by Lemma \ref{3.2} (e), we also have $p_A=k$. Thus $\|A^k\|=\|A\|^k=1$, which yields that $\|A_1\cdots A_k\|=1=\|A\|$ as required. Finally, for (c) $\Rightarrow$ (a), if $A$ is unitarily similar to $[a]\oplus A'$ with $|a|\ge\|A'\|$, then $w(A\otimes A)=w(A)$ by Lemma \ref{2.1}. On the other hand, if $A$ is permutationally similar to the block-shift matrix $A''$ with $\|A_1\cdots A_k\|=1$, then
\[ \|A^k\|=\|{A''}^k\|=\|A_1\cdots A_k\|=1=\|A\|^k.\]
Thus $p_A\ge k=n_A$. The equality $w(A\otimes A)=w(A)$ then follows from Theorem \ref{3.1}.
\end{proof}

\vspace{3mm}

\begin{corollary}\label{3.6}
Let $A=[a_{ij}]_{i,j=1}^n$, where $a_{ij}\ge 0$ for all $i$ and $j$, $a_{ij}=0$ for $i\ge j$, and $a_{i, i+1}>0$ for all $i$. Then the following conditions are equivalent:
\begin{enumerate}
\item[\rm (a)] $w(A\otimes A)=\|A\|w(A)$,
\item[\rm (b)] $p_A=n_A=n-1$, and
\item[\rm (c)] $a_{1 2}= \cdots =a_{n-1, n}$ and $a_{ij}=0$ for all other pairs of $i$ and $j$.
\end{enumerate}
\end{corollary}

\vspace{1mm}

\begin{proof}
In this case, $A$ is nonnegative, $\re A$ is irreducible and $n_A=n-1$. Consequently, Corollary \ref{3.5} yields the equivalence of (a), (b) and the condition (c') that $A$ is permutationally similar to a block-shift matrix $A''$ as in Corollary \ref{3.5} (c). Since $k=n_{A''}=n_A$ by Lemma \ref{3.2} (e), $A''$ is necessarily equal to $A$ with $|a_{1 2}\cdots a_{n-1, n}|=\|A\|$ and $a_{ij}=0$ for all other pairs of $i$ and $j$. The norm condition above yields that $a_{12}=\cdots=a_{n-1, n}=\|A\|$. Thus (c') is the same as (c), and we have the equivalence of (a), (b) and (c).
\end{proof}

\vspace{5mm}

\noindent
{\large Acknowledgements}

This research was partially supported by the National Science Council of the Republic of China under projects NSC-101-2115-M-008-006, NSC-101-2115-M-009-001 and NSC-101-2115-M-009-004 of the respective authors. P. Y. Wu was also supported by the MOE-ATU. This paper was presented by him at the 4th International Conference on Matrix Analysis and Applications in Konya, Turkey. He thanks the organizers for their works with the conference.

\end{document}